\documentclass{amsart}

\usepackage[a4paper, margin=4cm]{geometry}
\usepackage{amsfonts,amssymb,amscd}
\usepackage{hyperref}

\newtheorem{thm}{Theorem}[section]

\newtheorem{cor}[thm]{Corollary}

\theoremstyle{remark}

\newtheorem{remark}[thm]{Remark}

\renewcommand{\leq}{\leqslant}
\renewcommand{\geq}{\geqslant}
\renewcommand{\le}{\leqslant}
\renewcommand{\ge}{\geqslant}

\newcommand{\sph}{\mathbb{S}}
\newcommand{\rr}{\mathbb{R}}

\newcommand{\nn}{\mathbb{N}}
\newcommand{\hhh}{\mathcal{H}}

\newcommand{\ptl}{\partial}

\newcommand{\sg}{\sigma}
\newcommand{\Om}{\Omega}

\newcommand{\ga}{\gamma}
\newcommand{\la}{\lambda}

\newcommand{\escpr}[1]{\left<#1\right>}

\DeclareMathOperator{\vol}{vol}
\DeclareMathOperator{\cp}{Cap}
\DeclareMathOperator{\intt}{int}
\DeclareMathOperator{\ric}{Ric}
\DeclareMathOperator{\divv}{div}

\numberwithin{equation}{section}

\begin{document}

 \title[]{Comparison results for capacity}

\author[A.~Hurtado]{Ana Hurtado*}
\address{Departamento de Geometr\'{\i}a y Topolog\'{\i}a \\ Facultad
de Ciencias \\ Universidad de Granada \\ E--18071 Granada \\ Espa\~na}
\curraddr{}
\email{ahurtado@ugr.es}

\author[V.~Palmer]{Vicente Palmer$^{\#}$}
\address{Departament de Matem\`{a}tiques - Institut of New Imaging Technologies, Universitat Jaume I \\ E--12071 Castell\'{o}n \\ Espa\~na}
\email{palmer@mat.uji.es}

\author[M.~Ritor\'{e}]{Manuel Ritor\'{e}*}
\address{Departamento de Geometr\'{\i}a y Topolog\'{\i}a \\ Facultad
de Ciencias \\ Universidad de Granada \\ E--18071 Granada \\ Espa\~na}\email{ritore@ugr.es}

\thanks{* Supported by MICINN-FEDER grant MTM2010-21206-C02-01, and Junta de Andaluc\'{\i}a grants FQM-325 and P09-FQM-5088}

\thanks{$^{\#}$ Supported by the Caixa Castell\'{o} Foundation, and DGI grant MTM2010-21206-C02-02}

\subjclass[2010]{31C12, 31C15, 53C21, 58J65, 35J25}

\keywords{capacity, equilibrium potential, hyperbolicity, mean curvature, Cartan-Hadamard manifolds}
 
 \date{\today}

\begin{abstract}
We obtain in this paper bounds for the capacity of a compact set $K$. If $K$ is contained in an $(n+1)$-dimensional Cartan-Hadamard manifold, has smooth boundary, and the principal curvatures of $\ptl K$ are larger than or equal to $H_0>0$, then $\cp(K)\ge (n-1)\,H_0\vol(\ptl K)$. When $K$ is contained in an $(n+1)$-dimensional manifold with non-negative Ricci curvature, has smooth boundary,  and the mean curvature of $\ptl K$ is smaller than or equal to $H_0$, we prove the inequality $\cp(K)\le (n-1)\,H_0\vol(\ptl K)$. In both cases we are able to characterize the equality case. Finally, if $K$ is a convex set in Euclidean space $\rr^{n+1}$ which admits a supporting sphere of radius $H_0^{-1}$ at any boundary point, then we prove $\cp(K)\ge (n-1)\,H_0\mathcal{H}^n(\ptl K)$ and that equality holds for the round sphere of radius $H_0^{-1}$.
\end{abstract}

\maketitle

\thispagestyle{empty}

\section{Introduction}
The capacity of a compact set $K$ in a Riemannian manifold $M$ is defined by
\[
\cp(K)=\inf\bigg\{\int_{M} |\nabla \phi|^2\,dV :
\phi\in\mathcal{L}(K)\bigg\},
\]
where $\mathcal{L}(K)$ is the set of functions in the Sobolev
space $H_0^1(M)$ with $0\le\phi\le 1$ and $\phi\big|_K\equiv 1$, and $dV$ is the Riemannian volume in $M$.

From a physical point of view, the capacity of a compact set $K$
represents the total electric charge flowing into $M\setminus K$
through the boundary $\ptl K$. 
The exact value of the capacity of a set is known in few cases, and so its estimation in geometrical terms is of great
interest, not only in electrostatic, but in many physical
descriptions of flows, fluids, or heat, where Laplace equations is
used, \cite{conj-ps}. On the other hand, estimates of the capacity are enough to get geometrical
consequences such as the parabolic or hyperbolic character of the
manifold, \cite{I1}, \cite{I2}, \cite{MP1}, \cite{MP2}. We refer to the
survey by A. Grigor'yan \cite{Gri} for a discussion of these and
related concepts.

The employment of geometrical and
comparison techniques allows to obtain geometric inequalities
involving capacity, \cite{Sze1}, \cite{Sze2}, \cite[p.~17]{ps}. A bound for the capacity of a compact set with analytic boundary $K\subset\rr^3$ was obtained by Sz\"ego \cite[\S~2 (15)]{Sze1} in terms of the integral of the mean curvature $H$. He proved
\begin{equation}
\label{eq:szego1}
\cp(K)\le\int_{\ptl K}H\, dA,
\end{equation}
with equality for the round ball.

Sz\"ego also showed that, for a compact set $K\subset\rr^3$ with analytic boundary, one has
\[
\cp(K)\ge\bigg(\frac{3}{4\pi}\bigg)^{1/3}\!\!\!\!\vol(K)^{1/3},
\]
with equality precisely for the round ball, \cite[\S~2]{Sze2}, solving a problem stated by Poincar\'e in 1903 with an incomplete variational proof \cite[p.~vi]{ps}. In fact, Sz\"ego's proof works  for compact sets $K\subset\rr^{n+1}$, $n\ge 2$,  with smooth boundary, and provides the inequality
\begin{equation}
\label{eq:szegorn}
\cp(K)\ge c_{n+1}^2\,\frac{n-1}{n+1}\,\vol(K)^{(n-1)/(n+1)},
\end{equation}
with equality for the round ball. Here $c_{n+1}$ is the isoperimetric constant that appears in the optimal isoperimetric inequality $\vol(\ptl K)\ge c_{n+1}\vol(K)^{n/(n+1)}$ in $\rr^{n+1}$. Inequality \eqref{eq:szegorn} holds for any compact set $K\subset\rr^{n+1}$ by \cite[Thm.~8.1]{Gri}.

Polya and Sz\"ego also conjectured that there is a positive constant $\kappa$ such that
\[
\frac{\cp(K)}{\vol(\ptl K)^{1/2}}\ge\kappa,
\]
for any convex set $K\subset\rr^3$ (the above quantity has no lower bound for sets with non-convex boundary). They also conjectured that $\kappa=(32)^{1/2}\pi^{-1/2}$ and that equality holds for planar disks, see \cite{conj-ps}, \cite{ps}, \cite{Sze3}.

In this paper we prove two more results on estimations of the capacity of a compact set. In Theorem~\ref{thm2-b} we obtain
\begin{quote}
{\it Let $M^{n+1}$ be a Cartan-Hadamard manifold,
and $K\subset M^{n+1}$ a compact set with smooth boundary so
that the principal curvatures of $\ptl K$ are larger than or equal
to some constant $H_0>0$. Then we have
$$ \cp(K)\ge(n-1)\,H_0\vol(\ptl K).$$}
\end{quote}
Moreover, we are able to characterize the equality: it
is attained if and only  if the convex body has umbilical
boundary and $M\setminus\intt(K)$ is isometric to a warped
product. This result is optimal, as will be shown by Remark~\ref{rem:example}. As a consequence of this theorem, we obtain a simple proof of the
well-known fact that a Cartan-Hadamard manifold of dimension at
least three is hyperbolic.

On the other hand, for manifolds with non-negative Ricci curvature, we prove in Theorem~\ref{thm1-b}

\begin{quote}
{\it Let $M^{n+1}$ be a complete non-compact
Riemannian manifold with non-negative Ricci curvature, and
$K\subset M^{n+1}$ a compact set with smooth boundary. Assume that
the mean curvature of $\ptl K$ is smaller than or equal to
$H_0>0$. Then
\begin{displaymath}
\cp(K)\le (n-1)\,H_0\vol(\ptl K).
\end{displaymath}}
\end{quote}
We are also able to characterize the equality case. It is worthy to point out that for
analytic compact sets in $\rr^3$, the above result can be also obtained by applying the estimation \eqref{eq:szego1} in terms of the mean curvature given by Sz\"{e}go \cite{Sze1}.

In the proof of both results we use a transposition of the equilibrium potential of the Euclidean ball of radius $H_0^{-1}$ to the exterior of $K$ by means of the distance function to $K$. In the Cartan-Hadamard case there are no problems of differentiability, since the distance to a convex set with $C^\infty$ boundary is a $C^\infty$ function. In the non-negative Ricci curvature case, although the distance to $\ptl K$ is only smooth out
of the cut locus of $\ptl K$, the technical difficulties introduced by
the cut locus can be handled by a method of Cheeger and Yau
\cite{Ch-Y}. Theorems~\ref{thm2-b} and \ref{thm1-b} were proven by Ichihara \cite{I2}, \cite{I1} for geodesic balls, see also \cite[\S~15]{Gri}.

Theorems~\ref{thm2-b} and \ref{thm1-b} are valid in the Euclidean space $\rr^{n+1}$. In both results equality holds when $K$ is a Euclidean ball. Assuming $K$ is convex, Minkowski's formula implies the inequality $\vol(\ptl K)\ge\,(n+1)\,H_0\vol(K)$ when the principal curvatures of $\ptl K$ are larger than or equal to $H_0>0$, and it implies the opposite inequality when $H\le H_0$. We conclude in Corollary~\ref{cor:1} that
\[
\cp(K)\ge (n^2-1)\,H_0^2\vol(K),
\]
when the principal curvatures of $\ptl K$ are larger than or equal to $H_0$, and we conclude the opposite inequality in Corollary~\ref{cor:2} when the mean curvature of $\ptl K$ is smaller than or equal to $H_0$. In both Corollaries, equality holds for the round sphere.

In Euclidean space we can define a weak notion of ``principal curvatures bounded from below". A compact set $K$ in Euclidean space is $\la$-convex, $\la>0$, if there is a supporting ball of radius $\la^{-1}$ at every boundary point of $K$. A set $K$ with smooth boundary is $\la$-convex if and only if the principal curvatures of the boundary are larger than or equal to $\la$ \cite[p.~502]{MR0353225}, and so $\la$-convexity is an extension of the inequalities $\kappa_i\ge\la$, where $\kappa_i$ are principal curvatures, in a weak sense. We are able to extend the estimate in Theorem~\ref{thm2-b} to $H_0$-convex sets in Euclidean space without assuming the smoothness of the boundary. We prove in Theorem~\ref{baja regularidad}
\begin{quotation}
{\it Let $K$ be an $H_0$-convex body in $\mathbb{R}^{n+1}$, $H_0>0$. Then
\[
\cp(K)\ge(n-1)\,H_0\,\mathcal{H}^n(\ptl K),
\]
where $\mathcal{H}^n$ is the $n$-dimensional Hausdorff measure. Equality holds if and only if $K$ is a round ball of radius $H_0^{-1}$.}
\end{quotation}

We have organized this paper into three sections apart from this introduction. Section $2$ is devoted to establish
the definitions and results related to the capacity of a compact
set, which we will need in the sequel. In Section~$3$, we formulate
and prove our main results, Theorems~\ref{thm2-b} and \ref{thm1-b}, for the capacity of compact sets in
manifolds with a suitable control of their curvatures. Finally, in
Section~$4$, we state some consequences of Theorems~\ref{thm2-b} and \ref{thm1-b} in
the Euclidean space and we prove Theorem~\ref{baja regularidad}.

\section{Preliminaries}
Given an open set $\Om$ in a Riemannian manifold $M$ and a compact
set $K \subset \Om$, we define the capacity of $K$ in $\Om$ as
\begin{equation}\label{capacity}
\cp(K,\Om)= \inf \bigg\{\int_\Om |\nabla \phi |^2\, dV:
\phi \in \mathcal{L}(K, \Om)\bigg\},
\end{equation}
where $\mathcal{L}(K, \Om)$ is the set of functions on $M$ with
compact support in $\bar \Om$ which are locally Lipschitz and
satisfy: $0\leq \phi \leq 1$ and $\phi_{|K}=1$, see \cite{Gri}. Here $dV$ is the~Riemannian volume of $M$.

When $\Om$ is a precompact set and both $\Om$ and $K$ have smooth boundary, the infimum in (\ref{capacity})
is attained by the unique solution of the Dirichlet
problem in $\Om \setminus K$
\begin{equation}\label{eqDir}
\begin{cases}
\Delta u = 0\,\,\,&\text{on\, $\Om \setminus K$},\\
\phantom{\Delta }u = 1\,\,\,&\text{on\, $\partial K$}, \\
\phantom{\Delta }u = 0\,\,\,&\text{on\, $\partial \Om$}.
\end{cases}
\end{equation}
The function $u$ is called the equilibrium potential of $(K,
\Om)$. Using Green's formulae, we obtain
\begin{equation*}
\cp(K,\Om)= \int_{\Om\setminus K} | \nabla u |^2\, dV =
\int_{\partial K} \escpr{\nabla u,\nu}\, dA = \int_{\partial K}
| \nabla u |\, dA,
\end{equation*}
where $\nu$ is the unit normal vector field along $\partial K$
pointing into $K$, and $dA$ is the~Riemannian area element of $\ptl K$.

The capacity can be defined in the whole manifold $M$ by considering any exhaustion sequence $\{\Om_n\}_{n \in
\nn}$ covering $M$ such that $K \subset \Om_n\subset\overline{\Om}_n\subset \Om_{n+1}$
for all $n\in\nn$. Then,
\begin{displaymath}
\cp(K)=\lim_{n\rightarrow \infty} \cp(K, \Om_n).
\end{displaymath}

Moreover, if $\{\Om_n\}_{n\in\nn}$ is an exhaustion by precompact
sets such that $K\subset \Om_n$ for all $n\in\nn$, the maximum
principle for elliptic operators implies that $u_{n+1}\ge u_n$ in
$\Om_n\setminus K$. Hence the limit
\[
u=\lim_{n\to\infty} u_n
\]
exists and is a harmonic function with $u=1$ on $\ptl K$. A second
application of the maximum principle implies that $u$ is
independent of the exhaustion by precompact sets considered. It
follows easily that
\[
\cp(K)=\inf\bigg\{\int_{M} |\nabla \phi|^2\,dV :
\phi\in\mathcal{L}(K)\bigg\},
\]
where $\mathcal{L}(K)$ is the set of functions in the Sobolev
space $H_0^1(M)$ with $0\le\phi\le 1$ and $\phi\big|_K\equiv 1$.
The function $u$ is called the \emph{equilibrium potential} of $K$
and satisfies
\begin{equation}\label{eq:defcapint}
\cp(K)=\int_{\ptl K} |\nabla u|\,dA.
\end{equation}

For compact sets with non-smooth boundary, the equilibrium
potential $u$ can be obtained as the unique limit of the
equilibrium potentials $u_n$ of a nested sequence $\{K_n\}$ of
approximating smoothly bounded compact sets. Then $u$ is a
harmonic $C^{\infty}$ function in $M\setminus K$.

If $K$ is a convex body in Euclidean space, the equilibrium potential $u$
extends~continuously $\mathcal{H}^n$-almost everywhere to $\ptl
K$ \cite{Da}, and the gradient of $u$ has a non-tangential
limit to $\ptl K$ except in a set of $\mathcal{H}^n$-measure zero
\cite{MR1395668}. Also from \cite{MR1395668} it follows that~formula
\eqref{eq:defcapint} holds in the non-smooth case.

 For comparison purposes we are going to compute explicitly
the capacity and the equilibrium potential of the
$(n+1)$-dimensional closed ball $\bar{B}_{1/H_0}\subset\rr^{n+1}$
of mean curvature $H_0>0$.

By symmetry, we easily see that the equilibrium potential of
$(\bar{B}_{1/H_0},B_{(1/H_0) +t})$ only depends on the distance function $r$
to $\ptl\bar{B}_{1/H_0}$ and is given, for $n\ge 2$, by
\begin{equation}
\label{sol}
\Phi_t(r)=\frac{1}{1-(1+tH_0)^{1-n}}((1+rH_0)^{1-n}-
(1+tH_0)^{1-n}).
\end{equation}
In fact,
\begin{equation}\label{laplacEuclideo}
\Delta \Phi_t= \Phi_t''(r)+ \frac{n H_0}{1+r H_0}\,
\Phi_t'(r)=0,
\end{equation}
and the boundary conditions $\Phi_t(0)=1$ and $\Phi_t(t)=0$ are
satisfied. So, $\Phi_t$ is a solution of the corresponding
Dirichlet problem in $B_{(1/H_0) +t}\setminus \bar{B}_{1/H_0}$ given by
\eqref{eqDir}. Moreover, the maximum principle for elliptic
operators guarantees that $\Phi_t$ is the only solution.

Taking limits in \eqref{sol} when $t\to\infty$ we obtain that
the equilibrium potential of $\bar{B}_{1/H_0}$ is given by
\begin{equation}
\label{eq:epball}
\Phi_{H_0}(r)=(1+r\,H_0)^{1-n}.
\end{equation}

As a consequence of \eqref{eq:defcapint} and equality $|\nabla\Phi_{H_0}|=-\Phi'_{H_0}(0)=(n-1)\,H_0$ in $\ptl\bar{B}_{1/H_0}$ we get
\begin{equation}
\label{capball}
\cp(\bar{B}_{1/H_0})=(n-1)\,H_0\vol(\ptl\bar{B}_{1/H_0}).
\end{equation}
Formula \eqref{capball} also holds for $n=1$, since in this case $\Phi_t(r)=1-\log(1+H_0 r)/\log(1+H_0 t)$, $\Phi(r)\equiv 1$, and $\cp(\bar{B}_{1/H_0})=0$. 

\section{Main results}

We begin this section with a comparison result for the capacity of
a convex body in a Cartan-Hadamard manifold. In the sequel, by a
convex body we mean a compact convex body set with nonempty
interior. By a result of Alexander \cite[Thm.~1]{MR0448262}, a compact set with smooth boundary and positive principal curvatures in a Cartan-Hadamard manifold is a convex set.

\begin{thm}
\label{thm2-b} Let $M^{n+1}$ be a Cartan-Hadamard manifold, and
$K\subset M^{n+1}$ a compact set with smooth boundary. Assume that
the principal curvatures of $\ptl K$ are larger than or equal to
some constant $H_0>0$. Then we have
\begin{equation}
\label{eq:cp1} \cp(K)\ge(n-1)\,H_0\vol(\ptl K).
\end{equation}
Equality holds in \eqref{eq:cp1} if and only if $\ptl K$ is
totally umbilical with mean curvature~$H_0$ and $M\setminus
\intt(K)$ is isometric to the product $\ptl K\times
[0,\infty)$ with the warped metric $(1+H_0r)^2h+dr^2$, where $h$ is the Riemannian metric in $\ptl K$ and $r\in [0,\infty)$.
\end{thm}

\begin{proof}
Consider the equilibrium potential $\Phi=\Phi_{H_0}$ of the
Euclidean ball $\bar{B}_{1/H_0}$. Let $r:M\setminus\intt(K)\to\rr$ denote the distance to
$K$ \cite{gbs}. The function $r$ is $C^\infty$ in $M\setminus K$ since the cut locus of a convex set in a Cartan-Hadamard manifold is empty. We define $v:M\setminus\intt(K)\to\rr$ by
\[
v(p)=\Phi_{H_0}(r(p)).
\]
It is clear that $0\le v\le 1$, that $v\equiv 1$ in $\ptl K$, and that function $v$ is $C^\infty$ in $M\setminus K$. We have
\begin{equation}\label{eq:laplav}
\Delta v(p)=\Phi''(r(p))+\Phi'(r(p))\,nH_r(p),
\end{equation}
where $H_r$ is the mean curvature of $\ptl K_r$.

We estimate the mean curvature $H_r$ as in Rauch's Comparison Theorem \cite[II.6.4]{chavel-rg}. Let $\ga:[0,\infty)\to M$ be a geodesic leaving $\ptl K$ orthogonally and $E$ a Jacobi field along $\ga$ orthogonal to $\ga'(r)$ for all $r\ge 0$. Define
\begin{equation}
\label{eq:deff}
f(r)=\frac{\escpr{E',E}(r)}{|E(r)|^2},
\end{equation}
where $E'$ is the covariant derivative of $E$ along $\ga$. Computing the derivative of $f(r)$ using the Jacobi equation we have
\begin{equation}
\label{eq:estimatef}
f'(r)=-\frac{R(E,\ga',\ga',E)}{|E|^2}(r) +\frac{|E'|^2|E|^2-\escpr{E',E}^2}{|E|^4}(r)-f(r)^2\ge -f(r)^2,
\end{equation}
where in the inequality we have used that the sectional curvature $R(E,\ga',\ga',E)$ of the plane generated by $E(r)$ and $\ga'(r)$ is non-negative, and that $\escpr{E',E}^2\le |E|^2|E'|^2$ by Cauchy-Schwarz inequality. Integrating the above differential inequality, we obtain
\begin{equation}
\label{eq:estimate}
f(r)\ge\frac{f(0)}{1+f(0)r}.
\end{equation}
Equality holds in \eqref{eq:estimate} if and only if the sectional
curvature along the plane generated by $\ga'(r)$ and $E(r)$ is $0$
for all $r\ge 0$ and $E/|E|$ is a parallel vector field along
$\ga$ (observe that $(|E|^{-1}\,E)'$ has modulus $|E|^{-2}\big(|E|^2|E'|^2-\escpr{E,E'}^2\big)^{1/2}$.

Fix now some $r>0$. To estimate the mean curvature $H_r(p)$ we choose Jacobi fields $E_1,\ldots,E_n$ along $\ga$ which are
independent eigenvectors of the second fundamental form of $\ptl
K_r$ at $p$. Hence they are everywhere orthogonal to $\ga'$. For $i=1,\ldots,n$ we
consider the functions along $\ga$ defined by
\[
f_i=\frac{\escpr{E_i',E_i}}{|E_i|^2}.
\]
So we have
\begin{equation}
\label{eq:estimatehr}
nH_r(p)=\sum_{i=1}^n\frac{\escpr{E_i',E_i}}{|E_i|^2}(r(p))\ge
\sum_{i=1}^n \frac{f_i(0)}{1+f_i(0)r(p)}\ge \frac{nH_0}{1+H_0r(p)}.
\end{equation}
The first inequality in \eqref{eq:estimatehr} follows from
\eqref{eq:estimate} and the second one since $f_i(0)\ge H_0$ and
the function $x\mapsto x/(1+rx)$ is increasing. Equality holds in
\eqref{eq:estimatehr} for all $i=1,\ldots,n$ if and only if the principal curvatures of
$\ptl K$ at $\ga(0)$ are all equal to $H_0$, the sectional
curvatures of the planes generated by $\ga'(r)$ and $E_i(r)$ are
identically $0$, and $E_i/|E_i|$ is a parallel vector field. Assuming $|E_i(0)|=1$ for all $i=1,\ldots,n$, we get
\begin{equation}
\label{eq:parallel}
E_i(r)=(1+H_0r)\,P_i(r),
\end{equation}
where $P_i(r)$ is a parallel vector
field along $\ga$. This follows since $|E_i(r)|=1+H_0r$, because $(|E_i|^2)'=2\,\escpr{E_i',E_i}=2f\,|E_i|^2$ and $f(r)=H_0(1+H_0r)^{-1}$ as we are assuming that  equality holds in \eqref{eq:estimatehr}.  Assume that \eqref{eq:parallel} holds for every point in $\ptl
K$. Consider the map $f:\ptl K\times [0,\infty)\to M\setminus\intt(K)$
given by $f(p,r):=\exp_p(rN_p)$, where $N_p$ is the outer unit
normal to $\ptl K$. The pullback of the Riemannian metric of
$M\setminus\intt(K)$ can be written on $\ptl K\times [0,\infty)$ as the warped metric
\[
(1+H_0r)^2\,h+dr^2,
\]
where
$h$ is the Riemannian metric of $\ptl K$.

Let $u$ be the equilibrium potential of $K$. Since $\Phi' \leq 0$,
we conclude from \eqref{laplacEuclideo}, \eqref{eq:laplav} and
\eqref{eq:estimatehr} that
\begin{equation}
\label{eq:estimatelaplacian} \Delta v(p)\leq
\Phi''(r(p))+\Phi'(r(p))\,\frac{nH_0}{1+H_0r(p)}=0=\Delta u(p).
\end{equation}

Let us check that $v\ge u$ in $M\setminus K$. For every $t>0$, let
$u_t$ be the equilibrium potential of $(K,K_t)$, $\Phi_t$ the
equilibrium potential in Euclidean space of $(\bar{B}_{1/H_0},
B_{(1/H_0)+t})$, and $v_t:=\Phi_t\circ r$. Equation
\eqref{eq:estimatehr} implies that $\Delta (v_t-u_t)\le 0$ in $K_t
\setminus K$. By the maximum principle, $v_t\ge u_t$ in
$K_t\setminus K$. Since $v_t$, $u_t$ are increasing families
converging to $v$ and $u$ respectively we conclude that $v\ge u$
in $M\setminus K$.

As $\Delta(v-u)\le 0$ in $M\setminus K$, $v\ge u$ in $M\setminus
K$, and $u\equiv v\equiv 1$ in $\ptl K$, the maximum principle \cite[\S~3.2]{gt}
implies
\[
|\nabla v|\le |\nabla u| \quad\text{on }\ptl K,
\]
and so
\begin{align}
\label{eq:capestimate} \cp(K)=\int_{\ptl K} |\nabla
u|\,dA&\ge\int_{\ptl K}|\nabla v|\,dA
\\
\notag &= -\Phi'(0)\vol(\ptl K)=(n-1)\,H_0\vol(\ptl K).
\end{align}
Assume equality holds in \eqref{eq:capestimate}. Then $|\nabla v|=|\nabla u|$ on $\ptl K$ and, by the strong~maximum principle \cite{gt},  $u\equiv v$ on $M\setminus K$. This implies that equality also holds in \eqref{eq:estimatelaplacian} and \eqref{eq:estimatehr}. By the discussion of equality after \eqref{eq:estimatehr} we have that the principal curvatures of $\ptl K$ are all equal to $H_0$ and $M\setminus \intt(K)$ is isometric to the product $\ptl K\times [0,\infty)$, with the warped metric $(1+H_0r)^2h+dr^2$, where $h$ is the Riemannian metric in $\ptl K$ and $r\in [0,\infty)$.
\end{proof}

\begin{remark}
\label{rem:example}
The characterization of equality in Theorem~\ref{thm2-b} is the
best possible. Consider a smooth function $g:[0,\infty)\to
[0,\infty)$ so that $g(0)=0$, $g'(0)=1$, $g^{(2k)}(0)=0$ for all $k\in\nn$, $g''(t)\ge 0$ for all
$t\ge 0$. Then $\sph^n\times [0,\infty)$ with the warped metric $g(t)^2h_0+dt^2$, where $h_0$ is the standard Riemannian metric in $\sph^n$, is a smooth complete
$(n+1)$-dimensional Hadamard manifold $M$ 
\cite[\S~3.2.3]{petersen}. Given $t_0>0$, $H_0>0$, the function
$g$ can be chosen so that we have $H_0=g'(t_0)/g(t_0)$ and
\[
g(t_0+r)=g(t_0)\,(1+H_0r),\quad r\ge 0.
\]
Consider the convex ball $K:=\{t\le t_0\}$. The principal curvatures of $\ptl K$ are all equal to $H_0$, and the mean curvature of $\ptl K_r=\{t=t_0+r\}$ is given by
\[
H_r=\frac{g'(t_0+r)}{g(t_0+r)}=\frac{H_0}{1+H_0r},
\]
so that the equilibrium potential of $K$ is given by $u(r)=\Phi_{H_0}(r)$, where $\Phi_{H_0}$ is defined by \eqref{eq:epball}.
Then we have
\[
\cp(K)=\int_{\ptl K} |\nabla u|\,dA=-\Phi_{H_0}'(0)\,\vol{\ptl K}=(n-1)\,H_0\vol(\ptl K).
\]
In this example, the Riemannian metric inside $K$ can be slightly perturbed (around a point with strictly negative sectional curvatures) to a non-warped metric with non-positive sectional curvatures.
\end{remark}


\begin{remark}
The inequality of Theorem \ref{thm2-b}, can be also written as
\begin{displaymath}
\frac{\cp(K)}{\vol(\partial K)}\ge
\frac{\cp(\bar{B}_{1/H_0})}{\vol(\ptl \bar{B}_{1/H_0})},
\end{displaymath}
using \eqref{capball} whenever $K\subset M$ is a compact set with smooth boundary and principal curvatures satisfying $\kappa_i\ge H_0$. So we have obtained a comparison result
between the capacity of a convex body in $M$ and the capacity of a
round ball in the Euclidean space via the previous comparison of
the Laplacian of the distance function in both manifolds.
\end{remark}

\begin{remark}
A Riemannian manifold $M$ is said to be hyperbolic if there exists
a non-constant positive superharmonic function on $M$. Otherwise
it is called parabolic. The so-called type problem for manifolds is related to the problem of establishing necessary and
sufficient geometric conditions for a Riemannian manifold to be
hyperbolic or parabolic. This classical problem began to be
studied for Riemannian surfaces in the thirties by Ahlfors,
Myrberg, Nevanlinna and Royden among others and have given rise to
a large literature.

Lyons and Sullivan \cite{LS} gave a list of
equivalent conditions to check the hyperbolicity of an oriented
Riemannian manifold, which is known as the {\it
Kelvin-Nevanlinna-Royden criterium}. This criterium states that $M$
is hyperbolic if and only if there exists a compact set $K$ in $M$
with positive capacity.

Notice that, as a consequence of the above theorem, the capacity of a geodesic
ball of a Cartan-Hadamard manifold of dimension greater than or
equal to three is strictly positive. So, applying the
Kelvin-Nevanlinna-Royden criterium we have an alternative proof of
the hyperbolicity of such manifolds. This was previously known by the works of Ichihara \cite{I2}, \cite{I1}, see also \cite[Thm.~15.3]{Gri}. Alternatively, one can also use the isoperimetric inequality of Hoffman and Spruck \cite{MR0365424} together with Theorem~8.2 in \cite{Gri} to prove the hyperbolicity of Cartan-Hadamard manifolds of dimension larger than or equal to three.
\end{remark}

Now we state a comparison result for complete non-compact manifolds with non-negative Ricci curvature.

\begin{thm}
\label{thm1-b} Let $M^{n+1}$ be a complete non-compact Riemannian
manifold with non-negative Ricci curvature, and $K\subset M^{n+1}$
a compact set with smooth boundary. Assume that the mean curvature
of $\ptl K$ is smaller than or equal to $H_0>0$. Then
\begin{equation}
\label{eq:cp2} \cp(K)\le (n-1)\,H_0\vol(\ptl K).
\end{equation}
Moreover, equality holds in \eqref{eq:cp2} if and only if $M\setminus\intt(K)$ is isometric to the product $\ptl K\times [0,\infty)$ with the warped metric $(1+H_0r)^2h+dr^2$, where $h$ is the Riemannian metric of $\ptl K$, and $r\in [0,\infty)$.
\end{thm}

\begin{proof}
As in the proof of Theorem~\ref{thm2-b}, we consider the function
$v(p)=\Phi_{H_0}(r(p))$, which is Lipschitz in
$M\setminus K$ and smooth on $M\setminus (K\cup C_{\ptl K})$,
where $C_{\ptl K}$ is the cut locus of $\ptl K$ in $M\setminus K$.
Recall that $C_{\ptl K}$ is a closed set of Riemannian measure
zero on $M\setminus K$. Moreover, if $p\in M\setminus (K\cup
C_{\ptl K})$, then there is a unique minimizing geodesic
connecting $p$ and $\ptl K$, which is entirely contained in
$M\setminus (K\cup C_{\ptl K})$. We have
\begin{equation}
\label{eq:deltav2}
\Delta v(p)=\Phi''(r(p))+\Phi'(r(p))\,nH_r(p).
\end{equation}

Consider a geodesic $\ga:[0,c(p))\to M$ minimizing the distance to
$\ptl K$ with $\ga(0)=p$, where $c(p)$ is the cut distance. The
derivative of the mean curvature $H_r$ of the parallel
hypersurface $\ptl K_r$ along $\ga$ is given by
\begin{equation}
\label{eq:hr'}
nH_r'=-\text{Ric}(\ga',\ga')-|\sg_r|^2\le -|\sg_r|^2\le -n H_r^2,
\end{equation}
where $|\sg_r|^2$ is the squared norm of the second fundamental form of $\ptl K_r$, i.e., the sum of the squared principal curvatures. Equality holds in \eqref{eq:hr'} if and only if $\ric(\ga',\ga')=0$ and
$|\sg_r|^2=nH_r^2$ (when $\ptl K_r$ is totally umbilical at
$\ga(r)$). Hence, from the differential inequality \eqref{eq:hr'}, we conclude
\begin{equation}
\label{eq:hest2}
H_r(\ga(r))\le \frac{H_0(p)}{1+H_0(p)\,r}\le \frac{H_0}{1+H_0r}.
\end{equation}
In case of equality in \eqref{eq:hest2}, we have $\ric(\ga',\ga')\equiv 0$ along $\ga$, $|\sg_r|^2=nH_r^2$, and $H_r=H_0(1+H_0r)^{-1}$. In particular, $\ptl K_r$ is totally umbilical with principal curvatures $H_0(1+H_0r)^{-1}$. Hence, if $E$ is a Jacobi field along the geodesic $\ga$ orthogonal to $\ga$ we have
\begin{equation}
\label{eq:e'}
E'(r)=\frac{H_0}{1+H_0r}\,E(r).
\end{equation}
It is easy to get from this equation that $E''=0$. By the Jacobi equation the sectional curvature $R(\ga',E,E,\ga')$ of the plane generated by $E$ and $\ga'$ is equal to $0$. Formula \eqref{eq:e'} implies that $E$ and $E'$ are linearly dependent. So we get $\escpr{E,E'}^2=|E|^2|E'|^2$ and we conclude, as in the proof of Theorem~\ref{thm2-b}, that $E/|E|$ is a parallel vector field along $\ga$ and that $|E(r)|=1+H_0r$. So it follows, as in the Cartan-Hadamard case, that $M\setminus\intt(K)$ is isometric to the product $\ptl K\times [0,\infty)$ with the warped metric $(1+H_0r)^2h+dr^2$, where $h$ is the Riemannian metric of $\ptl K$.

From \eqref{eq:deltav2} and \eqref{eq:hest2} we get
\begin{equation}\label{laplac}
\Delta v(p)\ge 0,\quad \text{in}\quad M\setminus (K\cup C_{\ptl
K}).
\end{equation}

We can prove that the above inequality is also true in $M\setminus
K$ in the sense of distributions. Indeed, let us show, following
the method of Cheeger and Yau \cite{Ch-Y} (see also \cite{Gri}),
that for all non-negative $\phi \in C_0^\infty(M\setminus K)$ we
have that

\begin{equation}\label{weaksense}
\langle\Delta v, \phi\rangle:=-\int_{M\setminus K} \escpr{\nabla v,
\,\nabla \phi} \,dV \geq 0.
\end{equation}

Given $p\in \ptl K$, let $c(p)$ be the cut distance (possibly
$\infty$). Since the cut locus $C_{\ptl K}$ is a closed set, the
function $c(p)$ is lower semi-continuous, so it can be obtained as
the limit of an increasing sequence $\{c_k(p)\}$ of smooth
positive functions. Now, let us define
$$V_k:=\{ \exp_p(rN_p):\, p\in
\partial K,\, 0<r < c_k(p)\},$$
where $N_p$ is the unit outer normal to $\ptl K$.

In this way, we obtain an increasing sequence $\{V_k\}$ of open
sets with smooth boundary such that $\bigcup_k V_k= M\setminus (K\cup
C_{\ptl K})$. Then, applying Green's formulae
\begin{align*}
\int_{V_k} \escpr{\nabla v,\nabla \phi}\, dV&=-\int_{V_k} \phi\, \Delta v\,
dV -\int_{\partial V_k} \phi \,\escpr{\nabla v, \nu_k}\, dA\\
&=-\int_{V_k} \phi\,\Delta v\, dV -\int_{\partial V_k} \phi\,
\Phi'\,\escpr{\nabla r, \nu_k}\, dA,
\end{align*}
where $\nu_k$ is the unit inner normal vector field along $\ptl
V_k$. We have also used $\phi\equiv 0$ in $\ptl V_k\cap\ptl K$. As $V_k=\{\exp_p(c_k(p)\,N_p):p\in\ptl K\}$ is a radial graph over $\ptl K$, $\nabla r$ is never tangent to $\ptl V_k$, and hence $\escpr{\nabla r,\nu_k}\neq 0$. As $\nu_k$ is the \emph{inner} unit normal to $\ptl V_k$, we have $\escpr{\nabla r,\nu_k}<0$ $(\nu_k$ and $\nabla r$ form an obtuse angle). Taking
into account \eqref{laplac} and $\Phi'\leq 0$, we conclude
\[
\int_{V_k} \escpr{\nabla v,\nabla \phi}\, dV \leq 0.
\]
Since $C_{\ptl K}$ has measure zero, taking $k \rightarrow
\infty$, we obtain \eqref{weaksense}.

As in the proof of Theorem~\ref{thm2-b} for every $t>0$, let $u_t$
be the equilibrium potential of $(K,K_t)$, $\Phi_t$ the
equilibrium potential in Euclidean space of $(\bar{B}_{1/H_0},
B_{(1/H_0)+t})$, and $v_t:=\Phi_t\circ r$. By the method exposed
above $\Delta (v_t-u_t)\geq 0$ in $K_t\setminus K$ in the sense of
distributions and applying the weak maximum principle, \cite[\S~3.1]{gt}, we get $v_t\leq u_t$ in $K_t\setminus K$. Since $v_t$, $u_t$ are
increasing families converging to $v$ and $u$ respectively we
conclude that $v\geq u$ in $M\setminus K$.

 Hence we
can argue as in the proof of Theorem~\ref{thm2-b} to conclude that
\[
|\nabla v|\ge |\nabla u| \quad\text{on }\ptl K,
\]
and so
\[
\cp(K)=\int_{\ptl K}|\nabla u|\,dA\le\int_{\ptl K}|\nabla
v|\,dA=-\Phi'(0)\vol(\ptl K)=(n-1)\,H_0\vol(\ptl K).
\]
Equality holds when $u=v$ in $M\setminus (\ptl K\cup C_{\ptl K})$. So equality holds in \eqref{eq:hr'} and \eqref{eq:hest2}. By the previous discussion of equality, this implies that $M\setminus\intt(K)$ is isometric to the product $\ptl K\times [0,\infty)$ with the warped metric $(1+H_0r)^2h+dr^2$, where $h$ is the Riemannian metric of $\ptl K$.
\end{proof}

\begin{remark}
Note that the inequality obtained for the capacity of $K$, can be
also written as
\begin{displaymath}
\frac{\cp(K)}{\vol(\partial K)}\le
\frac{\cp(\bar{B}_{1/H_0})}{\vol(\ptl\bar{B}_{1/H_0})},
\end{displaymath}
using \eqref{capball}.
\end{remark}

\section{The Euclidean Case}
Hadamard's Theorem \cite{28.0643.01} implies that a compact set with smooth boundary and strictly positive principal curvatures is a convex set.  As a consequence of Theorems~\ref{thm2-b} and \ref{thm1-b}
we get
\begin{cor}
\label{cor:euclidean2} Let $K\subset\rr^{n+1}$ be a convex body
with smooth boundary. Assume that the principal curvatures of
$\ptl K$ are larger than or equal to some constant $H_0>0$. Then
we have
\begin{equation}
\label{eq:cp1-e} \cp(K)\ge(n-1)\,H_0\vol(\ptl K).
\end{equation}
Equality holds in \eqref{eq:cp1-e} if and only if $K$ is a round
ball of radius $H_0^{-1}$.
\end{cor}

\begin{cor}\label{thm2}
Let $K\subset\rr^{n+1}$ be a compact set with smooth boundary.
Assume that the mean curvature of $\ptl K$ is smaller than or
equal to some constant $H_0>0$. Then we have
\begin{equation}
\cp(K)\le(n-1)\, H_0\,\vol(\ptl K),
\end{equation}
with equality if and only if $K$ is a round ball of radius
$H_0^{-1}$.
\end{cor}

When $K\subset\rr^{n+1}$ is convex, Minkowski formula can be used to relate the capacity of
$K$ to its volume. To show this, pick a point $p$ in the
interior of $K$, and consider the radial vector field $X$ with center
$p$. Let $N$ be the  unit inner normal to $\ptl K$. By the convexity of $K$ we have $\escpr{X,N}< 0$ in $\ptl K$. Let $X^\top$ be the tangent projection of the vector field $X$ to $\ptl K$. Minkowski formula
\begin{equation}
\label{eq:mink}
\int_{\ptl K} (1+H\escpr{X,N})\,dA=0
\end{equation}
is obtained by applying the divergence theorem to the vector field $X^\top$ in $\ptl K$. Under our assumptions we have either $H\le H_0$ or $H\ge H_0$, so that Minkowski formula \eqref{eq:mink} implies
\[
\vol(\ptl K)=-\int_{\ptl K} H\escpr{X,N}\,dA\ge H_0\,
\bigg(-\int_{\ptl K}\escpr{X,N}\,dA\bigg)
\]
in case $H\ge H_0$ or the opposite inequality in case $H\le H_0$. Since $\divv X=(n+1)$ in $\rr^{n+1}$, by the divergence theorem applied to $X$ in $K$ we have
\begin{equation}
\label{eq:minkh0}
\vol(\ptl K)\ge (n+1)\,H_0\vol(K)
\end{equation}
in case $H\ge H_0$ and the opposite inequality when $H\le H_0$.

So we conclude from Corollaries \ref{cor:euclidean2} and \ref{thm2}, and from inequality\eqref{eq:minkh0}

\begin{cor}
\label{cor:1}
Let $K\subset\rr^{n+1}$ be a convex body
with smooth boundary. Assume that the principal curvatures of
$\ptl K$ are larger than or equal to some constant $H_0>0$. Then
we have
\begin{equation}
\label{eq:cp1-e2} \cp(K)\ge(n^2-1)\,H_0^2\vol(K).
\end{equation}
Equality holds in \eqref{eq:cp1-e2} if and only if $K$ is a round
ball of radius $H_0^{-1}$.
\end{cor}

\begin{cor}
\label{cor:2}
Let $K\subset\rr^{n+1}$ be a convex body with smooth boundary.
Assume that the mean curvature of $\ptl K$ is smaller than or
equal to some constant $H_0>0$. Then we have
\begin{equation}
\cp(K)\le(n^2-1)\, H_0^2\,\vol(K),
\end{equation}
with equality if and only if $K$ is a round ball of radius
$H_0^{-1}$.
\end{cor}

\vspace{1em}

Now we consider a general convex body $K\subset\rr^{n+1}$,
possibly with non-smooth boundary. Given $\la>0$, we shall say
that $K$ is $\la$-convex if, for every $p\in\ptl K$, there is a
closed ball $B$ of radius $\la^{-1}$ such that $K\subset B$ and
$p\in\ptl B$. Taking outer parallel bodies to both $K$ and $B$ it
is immediate to check that the outer parallel body $K_r$ is
$(\la^{-1}+r)^{-1}$-convex whenever $K$ is $\la$-convex. If $K$
has $C^2$ boundary then $\la$-convexity is equivalent to the lower
bound $\kappa_i\ge\la$ for the principal curvatures $\kappa_i$ of
$\ptl K$.

If $K\subset\rr^{n+1}$ is a convex body, we will denote by
$\xi:\rr^{n+1}\setminus\intt(K)\to\ptl K$ the metric projection, which is a
contractive Lipschitz map \cite[\S~1.2]{sch}. The gradient of the
distance function $r:\rr^{n+1}\setminus K\to\rr$ is given by
\begin{equation}
\label{eq:nablar}
\nabla r(q)=\frac{q-\xi(q)}{r(q)},
\end{equation}
and so it is a locally Lipschitz vector field in
$\rr^{n+1}\setminus K$. Hence $r$ is a $C^{1,1}_{loc}$ function on
$\rr^{n+1}\setminus K$, and the gradient of $r$ and $\xi$ are
differentiable in the same set. By Rademacher's Theorem, $r$ is
twice differenciable $\hhh^{n+1}$-almost everywhere in $\rr^{n+1}\setminus K$.
From \eqref{eq:nablar}, if $\xi$ is differentiable at
$q\in\rr^{n+1}\setminus K$ then it is also differentiable along
the points~in the minimizing geodesic connecting $q$ and $K$. So
$\nabla r$ is differentiable $\mathcal{H}^n$-almost everwhere in the boundary
$\ptl K_r$, for any $r>0$, which implies that $\ptl K_r$ is a
$C^{1,1}$-hypersurface. The second fundamental form and the
principal curvatures are defined $\mathcal{H}^n$-a.e in $\ptl
K_r$. Moreover, if $K_r$ is $\mu$-convex then the principal
curvatures $\kappa_i$ of $\ptl K_r$ satisfy $\kappa_i\ge\mu$.



Principal curvatures can be defined for sets of positive reach
\cite{MR0110078}, \cite[\S~1]{zahle}, a class in which convex sets
are included. Given $K$, we consider the set of points $p\in\ptl
K$ where $\xi$ is differentiable in $\xi^{-1}(\{p\})$, which is a
set of full $\mathcal{H}^n$ measure in $\ptl K$ by the discussion
above. We intersect this set with the one of regular points, those
for which there is a unique supporting hyperplane, which has also
full $\mathcal{H}^n$ measure \cite[Thm.~2.2.4]{sch}. Then, for $p$
in the intersection of these sets, we define
\begin{equation}
\label{eq:defk}
\kappa_i(p)=\lim_{t\downarrow 0}\frac{(\kappa_i)_t(p+t\,n_p)}{1-t\,(\kappa_i)_t(p+t\,n_p)},
\end{equation}
where $n_p$ is the outer unit normal to $\ptl K$ in $p$, and
$(\kappa_i)_t$, $i=1,\ldots,n$, are the principal curvatures of
$\ptl K_t$. If $K$ is $\la$-convex, then $K_t$ is
$(\la^{-1}+t)^{-1}$-convex, and it follows from \eqref{eq:defk}
that $\kappa_i\ge\la$.



\begin{thm}
\label{baja regularidad} Let $K$ be an $H_0$-convex body in
$\mathbb{R}^{n+1}$ for some constant $H_0>0$. Then
\begin{equation}
\label{eq:cplow} \cp(K)\ge(n-1)\,H_0\,\mathcal{H}^n(\ptl K).
\end{equation}
Equality holds in \eqref{eq:cplow} if and only if $K$ is a round
ball of radius $H_0^{-1}$.
\end{thm}

\begin{proof}
We consider the function $v(p)=\Phi_{H_0}(r(p))$, where
$r:\rr^{n+1}\setminus\intt(K)\to\rr$ is the distance function to $K$ and $\Phi_{H_0}$
is the equilibrium potential of the Euclidean ball of radius $H_0^{-1}$ as
defined in \eqref{eq:epball}. The function $v$ is $C^{1,1}$ in
$\rr^{n+1}\setminus K$. Let $p\in\rr^{n+1}\setminus K$ be a point
such that the distance function $r$ is twice differentiable along
the straight line minimizing the distance from $p$ to $\ptl K$.
Along this line we have
\begin{equation}
\label{eq:laplav3}
\Delta v(p)=\Phi''(r(p))+\Phi'(r(p))\,nH_r(p),
\end{equation}
where $H_r$ is the mean curvature of the parallel hypersurface
$\ptl K_r$. Since $K_r$ is $(1/H_0 + r)^{-1}$-convex, we have at
the regular points
\begin{equation}
\label{eq:hr}
H_r\ge \frac{H_0}{1+H_0r}.
\end{equation}

Let $u$ be the equilibrium potential of $K$. Since $\Phi' \leq 0$,
we conclude from \eqref{eq:laplav3}, \eqref{eq:hr} and
\eqref{laplacEuclideo}, that
\begin{equation}
\label{eq:estimatelaplacian3}
\Delta v(p)\leq
\Phi''(r(p))+\Phi'(r(p))\,\frac{nH_0}{1+H_0r(p)}=0=\Delta u(p),
\end{equation}
at every point $p$ where $r$ is $C^2$.

Let us check that $v\ge u$ in $\rr^{n+1}\setminus K$. For every
$t>0$, let $u_t$ be the equilibrium potential of $(K,K_t)$,
$\Phi_t$ the equilibrium potential in Euclidean space of
$(\bar{B}_{1/H_0}, B_{(1/H_0)+t})$, and $v_t:=\Phi_t\circ r$. Equation
\eqref{eq:estimatelaplacian3} implies that $\Delta (v_t-u_t)\le 0$
in $K_t\setminus K$ in the sense of distributions. By the weak
maximum principle, $v_t\ge u_t$ in $K_t\setminus K$. Since $v_t$,
$u_t$ are increasing families converging to $v$ and $u$
respectively we conclude that $v\ge u$ in $\rr^{n+1}\setminus K$.

Let $\Gamma(p)=\{q\in \rr^{n+1}\setminus K :\,|p-q|<Cr(q)\}$ be
a non-tangential cone at $p\in \ptl K$. By \cite[Thm.~1.7]{MR1395668},
$\nabla u$ has a non-tangential limit at $\mathcal{H}^n$-almost every
point of the boundary of $K$, i.e., $\lim_{q\rightarrow p}
\nabla u(q)$, $q\in \Gamma(p)$, exists for $\mathcal{H}^n$-almost
every $p\in \ptl K$.

We pick $p\in\ptl K$ so that $\nabla u$ extends to $p$, $\nabla v$
is defined in $p$ and $u$ extends continuously to $p$, and
restrict them to an outer ball $B$ to $K$. Applying Hopf's
maximum principle (see \cite{gt} and \cite{dl}) to $B$ we get,
since $\Delta(v-u)\le 0$ in $B$, $v\ge u$ in $B$, $v-u \in
\mathcal{C}^0(\bar{B})$ and $u(p)=v(p)=1$ that
\[
|\nabla v|(p)\le |\nabla u|(p).
\]
Since formula \eqref{eq:defcapint} is valid for arbitrary convex
sets \cite{MR1395668}, we get
\begin{align}
\label{eq:capestimateeucl} \cp(K)=\int_{\ptl K} |\nabla
u|\,dA&\ge\int_{\ptl K}|\nabla v|\,dA
\\
\notag &= -\Phi'(0)\,\mathcal{H}^n(\ptl
K)=(n-1)\,H_0\,\mathcal{H}^n(\ptl K),
\end{align}
which yields the desired inequality. Assume equality holds in \eqref{eq:capestimateeucl}. Then $|\nabla
v|=|\nabla u|$ $\mathcal{H}^n$-almost everywhere on $\ptl K$. So if
$p\in\ptl K$ is such that $|\nabla v|(p)=|\nabla u|(p)$,
applying the maximum principle on an outer ball $B$ to $K$ we have
that $u\equiv v$ on $B$ and then  $u\equiv v$ on
$\rr^{n+1}\setminus K$.

So the level sets of $u$ and $v$, and their gradients, coincide in $\rr^{n+1}\setminus K$. The level sets of $v$ are the parallel hypersurfaces $\ptl K_r$, which are of class $C^{1,1}$, and the gradient of $v$ is equal to $-\Phi'(r)\,\nabla r$, which never vanishes in $\rr^{n+1}\setminus K$. Hence the level sets of $u$ are $C^\infty$ hypersurfaces since $u\in C^\infty(\rr^{n+1}\setminus K)$ and $\nabla u=\nabla v\neq 0$.

Since equality holds in \eqref{eq:capestimateeucl}, it is also attained in \eqref{eq:hr} and \eqref{eq:estimatelaplacian3}. In particular, the principal curvatures of $\ptl K_r$ are all equal to $(H_0^{-1}+r)^{-1}$ , what implies that $\{K_r\}_{r>0}$ are concentric balls of radius $H_0^{-1}+r$. Hence $\ptl K$ is a sphere of radius $H_0^{-1}$.
\end{proof}

\def\cprime{$'$} \def\cprime{$'$}
\providecommand{\bysame}{\leavevmode\hbox to3em{\hrulefill}\thinspace}
\providecommand{\MR}{\relax\ifhmode\unskip\space\fi MR }
\providecommand{\MRhref}[2]{%
  \href{http://www.ams.org/mathscinet-getitem?mr=#1}{#2}
}
\providecommand{\href}[2]{#2}


\end{document}